\documentclass[12pt, reqno]{amsart}
\usepackage{amssymb,latexsym,amsmath,amsfonts}
\usepackage{supertabular}
\usepackage{enumerate}
\usepackage[mathscr]{eucal}
\usepackage{enumitem}
\usepackage{hyperref}

\usepackage[nameinlink]{cleveref}

\allowdisplaybreaks

\numberwithin{equation}{section}
\theoremstyle{definition}
\newtheorem{definition}{Definition}[section]

\theoremstyle{remark}
\newtheorem{remark}[definition]{Remark}

\theoremstyle{plain}
\newtheorem{proposition}[definition]{Proposition}
\newtheorem{theorem}[definition]{Theorem}
\newtheorem{lemma}[definition]{Lemma}
\newtheorem{result}[definition]{Result}
\newtheorem{corollary}[definition]{Corollary}

\usepackage[usenames]{color}

\definecolor{Red}{rgb}{1,0,0}
\definecolor{Magenta}{rgb}{0.69,0,0.83}
\definecolor{Green}{rgb}{0.117,0.706,0.314}
\definecolor{Grey}{rgb}{0.8,0.8,0.8}
\definecolor{Blue}{rgb}{0.1,0.1,1}

\newcommand{\unitdisk}{\mathbb{D}}

\newcommand{\ball}{\mathbb{B}}

\newcommand{\koba}{\mathsf{k}}

\newcommand{\clos}[1]{\overline{#1}}
\newcommand{\cl}{\mathrm{cl}}
\newcommand{\lraw}{\longrightarrow}
\newcommand{\wt}{\widetilde}
\newcommand{\wh}{\widehat}
\newcommand{\range}{\mathsf{ran}}

\newcommand*{\defeq}{\mathrel{\vcenter{\baselineskip0.5ex \lineskiplimit0pt \hbox{\scriptsize.}\hbox{\scriptsize.}}}=}

\newcommand{\bdy}{\partial}
\newcommand{\OM}{\Omega}

\newcommand{\smoo}{\mathcal{C}}

\newcommand{\C}{\mathbb{C}} 
\newcommand{\R}{\mathbb{R}}

\newcommand{\posint}{\mathbb{Z}_{+}}

\begin{document}

\title[Horofunction compactifications]{Horofunction compactifications and local Gromov model domains}
\author{Vikramjeet Singh Chandel, Sushil Gorai, Anwoy Maitra and Amar Deep Sarkar}

\address{VSC: Department of Mathematics and Statistics, Indian Institute of Technology Kanpur,
Kanpur -- 208 016, India}
\email{vschandel@iitk.ac.in}
	
\address{SG: Department of Mathematics and Statistics, Indian Institute of Science Education 
and Research Kolkata, Mohanpur -- 741 246, India}
\email{sushil.gorai@iiserkol.ac.in}

\address{AM: Department of Mathematics, Indian Institute of Technology Madras,
Chennai -- 600 036, India}
\email{anwoy@iitm.ac.in}
    
\address{ADS: School of Basic Sciences, Indian Institute of Technology Bhubaneswar, Argul --
752 050, India}
\email{amar@iitbbs.ac.in}


		
\begin{abstract}
We explore the {\em horofunction compactification} of complete
hyperbolic domains in complex Euclidean space
equipped with the Kobayashi distance. 
We provide a sufficient condition under which, given a domain $\OM$ as above, the
identity map from $\OM$ to itself extends to an embedding of $\clos{\OM}$ into the horofunction
compactification of $(\OM,\koba_\OM)$, with $\koba_\OM$ denoting the Kobayashi distance on $\OM$.
Notably, this condition admits unbounded
domains that are {\em not} Gromov hyperbolic relative to the
Kobayashi distance.
We also provide a large class of planar hyperbolic domains satisfying the above condition.

\end{abstract}

\maketitle	
	
\section{Introduction and statement of results}\label{Sec:Intro}
Let $\OM\subset\C^d$ be a %complete
Kobayashi hyperbolic domain equipped with the Kobayashi distance 
$\koba_\OM$. 
In this article, we are concerned with the following question: for what sorts 
of
domains $\OM\subset\C^d$ can one identify the so-called {\em horofunction 
compactification} of 
$(\OM,\koba_\OM)$ with some appropriate {\em Euclidean} compactification of $\OM$? Answering this question 
for general $\OM$ is intractable; one can try to give an answer for various special classes of $\OM$.
Arosio et.\,al \cite[Corollary~4.5, Theorem~5.3]{AFGG} proved that for the following three classes of domains the Euclidean compactification, 
horofunction compactification, and Gromov compactification are the same: 
\begin{itemize}
\item Bounded strongly pseudoconvex domains with $\mathcal{C}^2$-boundary. 
\smallskip

\item Bounded convex domains in $\C^d$ with $\mathcal{C}^{\infty}$ smooth boundary of finite type in the 
sense of D'Angelo. 
\smallskip

\item Bounded convex domains $\OM$ for which the squeezing function $s_{\OM}(z)$
satisfies $$\lim_{z\to\bdy\OM}s_{\OM}(z)=1.$$ 
\end{itemize}
Here, the important point is that the above three classes
of domains $\OM$ are {\em Gromov model domains}, i.e., 
$(\OM, \koba_\OM)$ is Gromov hyperbolic and the Gromov boundary and the Euclidean boundary are 
identifiable. Our main 
motivation in this article is to produce a class of domains that are not Gromov hyperbolic but 
such that one can 
still identify the Euclidean boundary with a {\em part} of the horofunction boundary.
In what follows, given $\OM$ as above (not necessarily bounded), we shall denote by $\clos{\OM}^H$
the horofunction compactification of, and in case $(\OM,\koba_\OM)$ is Gromov hyperbolic, 
by $\clos{\OM}^G$ the Gromov compactification of, $(\OM,\koba_\OM)$. 
The Gromov boundary (in the case of Gromov hyperbolicity) and the horofunction boundary will be denoted by 
$\bdy\clos{\OM}^G$ and $\bdy\clos{\OM}^H$ respectively. The reader is referred to 
Section~\ref{S:Prelims} for a quick introduction to these objects. 
\smallskip

The class of domains that we are interested in is as follows:
\begin{definition}
A (possibly unbounded) hyperbolic domain $\OM\subset\C^d$ is
said to be a {\em local Gromov model domain} if, for every $\xi\in\bdy\OM$, there exists a 
bounded neighbourhood $U$ of $\xi$ such that $U\cap\OM$ is connected, complete Kobayashi 
hyperbolic, Gromov hyperbolic relative to its Kobayashi distance, and such that 
$\mathsf{id}_{U\cap\OM}$ extends to a homeomorphism from $\clos{U\cap\OM}^G$ to $\clos{U\cap\OM}$.
In other words, every boundary point $\xi\in\bdy\OM$ has a bounded 
neighbourhood $U$ such that $U\cap\OM$ is a {\em Gromov model domain}, to use the 
terminology of \cite[Definition~1.3]{BGNT2022}.
\end{definition}

\begin{remark} \label{rem:unb_loc_Grom_mod_not_Grom_hyp}
In \cite[Corollary~1.7]{BGNT2022} it is proved that if $\OM$ is a {\em bounded}
local Gromov model domain then it is a Gromov model domain. 
However, there exist unbounded local Gromov model domains
in dimension one that are not Gromov hyperbolic (and hence are not Gromov model
domains)\,---\,see, for example, \cite[Proposition~9.1]{BZ2023}. It is also
easy to construct unbounded domains in $\C^d$, $d\ge 2$, 
with smooth boundary that are
local Gromov model domains but not Gromov hyperbolic; such a construction is 
presented in brief at the beginning of \Cref{S:proofs}.
\end{remark}

For a bounded Gromov model domain $\OM$, given $x\in\bdy \OM$ and two geodesic rays
$\gamma$, $\sigma$ landing at $x$ (i.e., $\lim_{t\to\infty}\gamma(t)=x=\lim_{t\to\infty}\sigma(t)$), 
since $x$ represents a point in the Gromov boundary, by definition (see Section~\ref{S:Prelims}), 
$\gamma$ and $\sigma$ are asymptotic, i.e.,
\[ \sup_{t\in [0,\infty)} \koba_\OM(\gamma(t),\sigma(t)) < \infty. \]
Our first result says that an analogous result holds for local Gromov model domains. 

\begin{proposition} \label{prp:loc_Grom_mod_asymp}
Suppose that $\OM\subset\C^d$ is a Kobayashi hyperbolic 
local Gromov model domain that is 
hyperbolically embedded in $\C^d$, and $\gamma,\sigma$ are 
$\koba_\OM$-geodesic rays that land at the 
boundary point $\xi\in\bdy\OM$. Then $\gamma$ and $\sigma$ are 
$\koba_\OM$-asymptotic, i.e., 
\[ \sup_{t\in [0,\infty)} \koba_\OM(\gamma(t),\sigma(t)) < \infty. \]
\end{proposition}
\noindent Here, $\OM$ is said to be {\em hyperbolically embedded}
in $\C^d$ if for any $p\neq q\in\bdy\OM$, 
\[
\liminf_{(x, y)\to(p, q)}\koba_\OM(x, y)>0.
\]
This is not how the notion of being hyperbolically embedded is usually defined in the
literature, but is easily seen to be equivalent to the standard definition
(see, e.g., \cite[Definition~1.3]{Rumpa_2025_Vis_mani}). We use the definition above
because it matches exactly with the {\em boundary separation property} as defined and
used in \cite{CGMS2024} (see \cite[Subsection~2.4]{CGMS2024}), and which, as we just
observed, coincides with the notion being discussed.
Note that every bounded domain is hyperbolically embedded in $\C^d$. This follows,
for example, from \cite[Proposition~3.5]{Bharali_Zimmer} and the fact that the Kobayashi
distance is the integrated form of the Kobayashi--Royden metric (see 
\cite[Lemma~2.8]{CGMS2024} and the following discussion). It is also not difficult
to see that $\OM$ is hyperbolically embedded in $\C^d$ if and only if it is
{\em hyperbolic at each boundary point} (for a definition of this latter term, and
for the equivalence, see, for instance, \cite[Definition~2.1, Proposition~2.2]{Nik_Okt_Tho}).
\smallskip

Given a geodesic ray $\gamma$ in an appropriate distance space
$(X, d)$, there is a well-known construction that associates to 
$\gamma$ a point in the horofunction boundary 
of $(X, d)$. This is via the Busemann function associated 
to $\gamma$, $B_\gamma$; see Section~\ref{S:Prelims} for more details.
A basic question is: given
geodesic rays $\gamma$ and $\sigma$, when do $B_\gamma$ and $B_\sigma$ determine the 
same point of the
horofunction boundary? A natural condition explored in \cite{AFGG} is the notion of 
{\em strong asymptoticity} (see \cite[Definition~3.1]{AFGG}).  
Two $\koba_\OM$-geodesic-rays $\sigma, \gamma$ are said to be
{\em strongly asymptotic} if they satisfy 
\[ \exists\,T\in\R,\; \lim_{t\to\infty} \koba_\OM(\gamma(t),\sigma(t+T))=0. \]
Arosio et.\,al.~in \cite{AFGG} proved that for a bounded domain $\OM$ 
satisfying $\lim_{z\to\bdy\OM}s_{\OM}(z)=1$, where $s_{\OM}(z)$ denotes the squeezing function, 
any two asymptotic geodesic rays that land at the same boundary
point are strongly asymptotic. Our second result says that
an analogous result also holds for local Gromov model domains. 

\begin{theorem} \label{thm:loc_Grom_hyp_then_str_asymp}
Suppose that $\OM\subset\C^d$ is a Kobayashi hyperbolic (not necessarily bounded)
local Gromov model domain that is hyperbolically
embedded in $\C^d$. Suppose that $\OM$ admits squeezing functions, and that
the squeezing function tends to $1$ at the boundary, i.e., $\forall\,\xi\in\bdy\OM$,
$\lim_{\OM\ni z\to\xi} s_\OM(z)=1$. Then, for every $\xi\in\bdy\OM$ and every two 
$\koba_\OM$-geodesic-rays $\gamma,\sigma$ in $\OM$ landing at $\xi$, $\gamma$ and $\sigma$ are 
strongly asymptotic.
\end{theorem}

Before we state the next result, we wish to remind the reader of the so-called
{\em visibility property with respect to the Kobayashi distance}. The notion of
visibility has been studied intensively in recent years; see, for example,
\cite{Bharali_Zimmer}, \cite{Bharali_Maitra}, \cite{BNT2022}, \cite{CMS2023},
\cite{BGNT2022}, \cite{ADS2023}, \cite{Nik_Okt_Tho}, \cite{Banik},
\cite{Nik_Okt_Tho_Visib_C2_pscvx}, \cite{Rumpa_2025_taut}, and \cite{Rumpa_2025_Vis_mani}.
In this short note, we will not recall the definition of this property, but rather
refer the reader to \cite[{Sections~1\,\&\,{}3}]{CMS2023} for the
definitions and a brief discussion of the various types of visibility.
\smallskip

We are now in a position to state the main result of this paper, which 
provides a sufficient condition under which the natural map from $\OM$ to
$\clos{\OM}^H$ extends continuously to $\clos{\OM}$.

\begin{theorem}\label{thm:strongly_asymp_then_cont_xtnsn}
Let $\Omega \subset \C^d$ be a complete Kobayashi hyperbolic domain that has
the geodesic visibility property and satisfies:  
for every $\xi \in \bdy\Omega$, any two geodesic rays $\gamma$ and $\sigma$ in 
$(\Omega,\koba_\Omega)$ landing at $\xi$ are strongly asymptotic 
(with respect to $\koba_\Omega$).  
Then the map $\Psi:\clos{\Omega}\to \clos{\Omega}^H$ defined by
\[
\Psi(\xi) \defeq 
\begin{cases}
[\koba_\Omega(\,\cdot\,,\xi)], & \text{if } \xi\in\Omega,\\[4pt]
[B_\gamma], & \text{if } \xi\in\bdy\Omega \text{ and } \gamma 
\text{ is a geodesic ray in } (\Omega,\koba_\Omega) 
\text{ landing at } \xi,
\end{cases}
\]
is well defined and embeds $\clos{\Omega}$ into $\clos{\Omega}^H$.
\end{theorem}

\noindent
The above result is inspired by, and modeled on, \cite[Theorem~3.5]{AFGG}, but with key differences.  
While \cite[Theorem~3.5]{AFGG} concerns the embedding (indeed, the topological equivalence) of the 
Gromov compactification of a proper geodesic Gromov hyperbolic space into its horofunction 
compactification, the theorem above establishes an embedding of the closure of a domain 
in complex Euclidean space into its horofunction compactification relative to $\koba_\Omega$.  
In place of Gromov hyperbolicity, we assume the geodesic visibility property, which is more 
natural in the present setting.
\smallskip

We would like to present a simple example illustrating the fact that, in the context of the theorem
just stated, the natural map from $\OM$ to $\clos{\OM}^H$ (namely, $x\mapsto 
[\koba_\OM(\cdot,x)]$) {\em does not}, in general, 
extend continuously to $\clos{\OM}$, the Euclidean closure
of $\OM$. Take $\OM\defeq \unitdisk\setminus [0,1]$, the unit disk with the line segment from
$0$ to $1$ removed. It is not hard to see that, for every $p\in (0,1]$, $[\koba_\OM(\cdot,x)]$
has no limit as $\OM\ni x\to p$. This follows from three things: the classical behaviour of a 
Riemann map from $\unitdisk$ to $\OM$ (see, for instance, 
\cite[Proposition~2.5, Theorem~2.6]{Pomme});
the fact that any Riemann map is an isometry for the hyperbolic (equivalently, Kobayashi)
distance; and the standard
fact that $\clos{\unitdisk}^H$ is naturally homeomorphic to $\clos{\unitdisk}$.
\smallskip

Finally we present a class of planar hyperbolic domains, introduced recently in \cite{CGMS2024},
to which the above theorem applies.
These are local Gromov model domains; in particular, this class contains hyperbolic domains
that are not Gromov hyperbolic with respect to the Kobayashi distance, and domains whose 
boundaries could be very irregular (see \cite[Section~6]{CGMS2024}).

\begin{definition}\label{Def:cond1} A 
hyperbolic domain $\OM\subset\C$ is said to satisfy Condition~1 (or to be a Condition~1 domain)
if for any point $p\in\bdy \OM$, there exist an $r>0$ and
a topological embedding $\tau_p:\clos{\unitdisk}\lraw\clos{\OM}$ such that
$\tau_p(\unitdisk)\subset\OM$ and $D(p, r)\cap\OM\subset\tau_p(\unitdisk)$. 
\end{definition}

In the above, $D(p,r)$ denotes the open disc of radius $r$ centred at $p$. It is easy to see 
why Condition~1 domains 
are local Gromov model domains: take a point $p\in U$,
and let $\tau_p$, $r>0$ be as in the above definition. Note that $\tau_p(\unitdisk)$ is a 
Jordan domain; therefore, appealing to the Riemann Mapping Theorem and Carath\'{e}odory's Extension
Theorem \cite[Theorem~4.3]{BCM}, we may assume that $\tau_p$ restricted to $\unitdisk$ 
is a biholomorphism onto $\tau_p(\unitdisk)$ that extends to a homeomorphism from $\clos{\unitdisk}$ 
to $\clos{\tau_p(\unitdisk)}$. 
Clearly, $\tau_p(\unitdisk)$ equipped with its Kobayashi distance is a Gromov model domain.
\smallskip

Condition~1 domains were also completely characterized in \cite{CGMS2024}: these are domains for 
which the boundary is locally connected and each 
boundary component is a Jordan curve in the Riemann sphere \cite[Proposition~4.15]{CGMS2024}.
They also possess the geodesic visibility property. Regarding their squeezing function, 
we prove the following result: 
\begin{proposition} \label{prp:sq_fn_tend_1_bdry_cond1}
Suppose that $\OM\subset\C$ is a domain that satisfies Condition 1. Then, for every $p\in\bdy\OM$,
$\lim_{z\to p}s_\OM(z)=1$.
\end{proposition}
As a consequence of Theorem~\ref{thm:strongly_asymp_then_cont_xtnsn} and 
Theorem~\ref{thm:loc_Grom_hyp_then_str_asymp}, we get the follwing result. 
\begin{corollary}\label{Cor:}
Let $\OM$ be a Condition~1 domain. Then the identity map extends as an embedding 
from $\clos{\OM}$ to $\clos{\OM}^{H}$. 
\end{corollary}

\section{Preliminaries}\label{S:Prelims}
Let us recall the horofunction compactification and the horofunction boundary 
of a distance space. These notions were
introduced by Gromov in 1981 \cite{Gromov:1981} in the setting of ${\rm CAT(0)}$ spaces. 
Given a locally compact distance space $(X,d)$, we endow $C(X; \R)$, 
the space of real-valued continuous functions on $X$, with the topology of uniform convergence on compact 
sets. 
The strategy is to embed $X$ in $C(X; \R)$, as follows.
Choose and fix a point $p\in X$ and consider the mapping $\psi_p : X \to C(X; \R)$ given
by: $\forall\,x,y\in X,\;\psi_p(x)(y) \defeq d(x,y)-d(x,p)$. Note that, for every $x\in X$,
$\psi_p(x)$ is a continuous function on $X$ that vanishes at $p$; note also that $\psi_p$
itself is injective.
It follows easily from the Arzel{\`a}--Ascoli theorem
that $\psi_p(X)$ is a relatively compact subset of $C(X; \R)$; hence its closure $\clos{\psi_p(X)}$
in $C(X; \R)$ is compact. One could define $\clos{\psi_p(X)}$ to be the horofunction compactification
of $(X,d)$, but this would leave open the possibility that this compactification depends on $p$.
\smallskip

To remove this dependence on $p$, we proceed as follows. We let $A$ denote the set of constant
functions in $C(X,\R)$. Then $A$ is a closed subspace of $C(X,\R)$ and so $C_{*}(X,\R)\defeq
C(X,\R)/A$ is defined as a (quotient) topological vector space in its own right. For an arbitrary 
function $f\in C(X,\R)$, we let $[f]$ denote the image of $f$ under the natural projection from
$C(X,\R)$ onto $C_{*}(X,\R)$. Define $\tau: X\to C_{*}(X,\R)$ as follows: $\forall\,x\in X,$ $\tau(x)
\defeq [d(x,\cdot)]$. Here, $d(x, \cdot):X\lraw\R$ is the distance from $x$. 
If, for $p\in X$ arbitrary, we let $C_p(X,\R) \defeq \{f\in C(X,\R)\mid
f(p)=0\}$, then $C_p(X,\R)$ is a closed subspace of $C(X,\R)$. Consider the map $H_p:C(X,\R)\to
C_p(X,\R)$ given by $H_p(f)\defeq f-f(p)$; it is clear that $H_p$ is a continuous linear surjection
with kernel $A$. Thus, it establishes an isomorphism of topological vector spaces, say 
$\check{H}_p : C_*(X,\R) \to C_p(X,\R)$. Now note that the map $\psi_p$ defined above maps $X$
into $C_p(X,\R)$; in fact, as we will note below, in ``nice'' situations it is an {\em embedding} 
of $X$ into $C_p(X,\R)$. As mentioned above, $\clos{\psi_p(X)}$ is (a) compact (subset of $C_p(X,\R)$,
or, equivalently, of $C(X,\R)$). It makes sense to consider the map $\check{H}_p^{-1}\circ\psi_p:
X\to C_*(X,\R)$; it is (at least) a continuous injective map. The closure of its image is
$\clos{\check{H}_p^{-1}(\psi_p(X))}=\check{H}_p^{-1}(\clos{\psi_p(X)})$, which is a compact
subset of $C_*(X,\R)$. Furthermore, it is easy to check that for {\em any} $p\in X$,
$\check{H}_p^{-1}\circ\psi_p=\tau$. Therefore, for any $p,q\in X$, 
\[ \clos{\check{H}_p^{-1}\circ\psi_p(X)} = \clos{\check{H}_q^{-1}\circ\psi_q(X)} = \clos{\tau(X)}.\]
We define $\clos{\tau(X)}$ to be the {\em horofunction compactification} or {\em horocompactification}
of $(X,d)$, and we denote it by $\clos{X}^H$. 
We define $\clos{\tau(X)}\setminus
\tau(X)$ to be the {\em horofunction boundary} or {\em horoboundary} of $(X,d)$, and we denote it by
$\bdy^H X$. By a {\em horofunction} of $(X,d)$ we mean an 
element $h\in C(X,\R)$ such that $[h]\in \bdy^H X$. Given an element $\alpha\in\bdy^H X$ and given
$p\in X$, there exists a unique element $h\in C_p(X,\R)$ such that $[h]=\alpha$; this $h$ we denote by
$\alpha_p$.
We now state the following fact that is well-known in the literature; see e.g. 
\cite[Proposition~3.0.26]{Schilling}.
%(more precisely, of $(X,d)$). 
%Having checked that the definition of horofunction compactification is
%independent of the choice of base point, we largely return to calculating using a fixed but arbitrary
%base point; in effect this means calculating using $\psi_p$ instead of $\tau$. With such a choice
%of a fixed but arbitrary $p$, a horofunction becomes a particular kind of continuous function that
%vanishes at $p$, and the horoboundary a collection of such functions.

\begin{result} \label{prp:length-type_space_horofunc_embed}
Suppose that $(X,d)$ is a 
%locally compact, complete 
proper, geodesic distance space.
Then, for every $p\in X$, $\psi_p$ is an embedding of $X$ into $C(X,\R)$.
\end{result}

Let us now recall the Gromov compactification of a Gromov hyperbolic space.
Let $(X,d)$ be a proper, geodesic, Gromov hyperbolic metric space.  
A \emph{geodesic ray} in $X$ is an isometric embedding $[0,\infty)\to X$;  
a \emph{geodesic line} is an isometric embedding $\R\to X$; and  
a \emph{geodesic segment} (sometimes also referred to simply as a {\em geodesic})
is an isometric embedding $[a,b]\to X$ for $a,b\in\R$ with $a\le b$.
(It is also pertinent to recall the following definition: if $(X,d)$ is an arbitrary distance space
and if $\lambda\ge 1$ and $\kappa\ge 0$ are parameters, then by a $(\lambda,\kappa)$-quasi-geodesic
in $(X,d)$ we mean a map $\gamma:I\to X$, where $I\subset\R$ is an interval, such that
\[ \forall\,s,t\in I,\; (1/\lambda)|s-t|-\kappa \le \koba_{\OM}(\gamma(s),\gamma(t)) \le \lambda|s-t|+\kappa.)
\]
\smallskip

The \emph{Gromov boundary} $\bdy^G X$ is the set of equivalence classes of geodesic rays under the relation of being \emph{asymptotic}, denoted $\sim_G$, where
\[
\gamma \sim_G \sigma \iff \sup_{t\ge0} d(\gamma(t),\sigma(t))<\infty.
\]
If $X$ is proper, geodesic, and Gromov hyperbolic, the space
$\overline{X}^G := X \sqcup \bdy^G X$
admits a natural metrizable topology making it a compactification of $X$.  
Convergence in this topology is described as follows.
A sequence $(x_n)\subset X$ converges to a boundary point $\xi\in\bdy^G X$  
if there exists $p\in X$ such that, for every sequence of geodesic segments $\gamma_n$ joining $p$ to $x_n$, 
every subsequence of $(\gamma_n)$ admits a further subsequence converging uniformly on the
compact subsets of $[0,\infty)$ to a geodesic ray representing $\xi$.  
Similarly, a sequence $(\xi_n)\subset\bdy^G X$ converges to $\xi\in\bdy^G X$  
if there exists $p\in X$ such that, for any choice of geodesic rays $\gamma_n$ emanating from $p$ and 
representing $\xi_n$, every subsequence of $(\gamma_n)$ has a further subsequence converging uniformly on 
the compact subsets of $[0,\infty)$ to a ray representing $\xi$.
\smallskip

We now recall the Busemann function corresponding to a geodesic ray.
Suppose that $(X,d)$ is a proper, non-compact
(equivalently, unbounded)
geodesic distance space and $\gamma : [0,\infty) \to X$ is a geodesic ray in $(X,d)$.
Consider the family
\[ \Big( d\big(\cdot\,,\gamma(t)\big)-d\big(\gamma(t),\gamma(0)\big) \Big)_{t\ge 0} \]
of functions on $X$. Using the fact that $\gamma$ is a geodesic ray, it is easy to see that this
family of functions is pointwise decreasing and locally bounded from below. Therefore, by the
Arzel{\`a}--Ascoli theorem, the family converges, as $t\to\infty$, to an element
$B_{\gamma}(\cdot\,,\gamma(0))$ of $C(X,\R)$.
This element is a horofunction; indeed, returning to the terminology introduced above,
$B_\gamma(\cdot\,,\gamma(0)) = \lim_{t\to\infty} \psi_{\gamma(0)}(\gamma(t))$. 
In fact, using the observations made at the beginning of this section, it is not difficult to see that 
for $p\in X$ arbitrary, $\lim_{t\to\infty} \psi_{p}(\gamma(t))$ exists as
well; we denote it by $B_{\gamma}(\cdot\,,p)$. Thus, we have a function $B_\gamma : X\times X\to\R$
given by
\[ B_\gamma(x,y) = \lim_{t\to\infty} \Big( d\big(x,\gamma(t)\big)-d(\gamma(t),y) \Big); \]
this we call the {\em Busemann function} corresponding to the geodesic ray $\gamma$.
Given $x,y\in X$ and $t\in [0,\infty)$ arbitrary, note that 
$[\psi_x(\gamma(t))]=[\psi_y(\gamma(t))]$; taking the limit as $t\to\infty$, it follows that 
$[B_\gamma(\cdot\,,x)]=[B_{\gamma}(\cdot\,,y)]$; this common element of $C_{*}(X,\R)$ we denote by
$[B_\gamma]$. We now recall the following result alluded to in the Introduction. 

\begin{result}[{\cite[Proposition~3.3]{AFGG}}] \label{res:geod_str_asymp_Buse_eq}
Let $(X,d)$ be a proper geodesic distance space. If two geodesic rays $\gamma,\sigma$ in $X$
are strongly asymptotic, then $B_\gamma=B_\sigma$.
\end{result}

For domains possessing the geodesic visibility property, we have the following result (in what follows, 
given $\OM$, a possibly unbounded, complete Kobayashi hyperbolic domain, $\clos{\OM}^{End}$ denotes 
the {\em end compactification} of $\clos{\OM}$\,---\,see, for instance, \cite[Subsection~2.1]{CGMS2024}
for a quick introduction to this):
\begin{result} \label{res:id_extends_horo_to_end_visib}
Suppose that $\OM\subset\C^d$ is a complete Kobayashi hyperbolic domain that has the 
geodesic visibility property. Then the identity map $\mathsf{id}:\OM\to\OM$ extends 
to a continuous surjective map from $\clos{\OM}^H$ to $\clos{\OM}^{End}$. Moreover, 
given $\xi\in\bdy\clos{\OM}^{End}$, let
\[ 
 \mathscr{F}_{\xi} \defeq \{ \alpha \in C_{*}(\OM,\R) \mid \exists\, (x_n)\subset\OM
  \text{ such that } x_n\to\xi \text{ and } \tau(x_n)\to\alpha\}, 
 \]
where $\tau$ is as in the definition of horofunction compactification. 
Then, for every $\xi_1,\xi_2\in\bdy\clos{\OM}^{End}$ with $\xi_1\ne\xi_2$, $\mathscr{F}_{\xi_1}\cap 
\mathscr{F}_{\xi_2}=\emptyset$.
\end{result}
\noindent We skip the proof of this result here as it follows from 
more general results in the recent article \cite{CM:2025}. 
However, in our setting, we present the main idea: an important role is played 
by the following result. 
\begin{result}[{\cite[Corollary~4.3]{Web-Win-Boundaries-hyperbolic}}] 
\label{res:if_conv_same_horo_then_Grom_equiv}
Suppose $(X,d)$ is a proper, geodesic distance space. Then, for every $\alpha\in\clos{\tau(X)}\setminus
\tau(X)$, and every two sequences $(x_n)$ and $(y_n)$ in $X$ such that $(\tau(x_n))$ and $(\tau(y_n))$ 
both converge to $\alpha$, one has $\lim_{n,m\to\infty} (x_n|y_m)_o = \infty$.
\end{result}
Following the terminology in \cite[Definition~6.2]{Bharali_Zimmer}, the above result
essentially says that $\clos{X}^{H}$ is a {\em good compactification}.
The continuous extension of the identity map in 
\Cref{{res:id_extends_horo_to_end_visib}} then follows from ideas similar to those 
in the proof of \cite[Theorem~6.5]{Bharali_Zimmer}.
Alternatively, one could also appeal to 
\cite[Theorem~1.4]{CM:2025}. The second assertion in \Cref{res:id_extends_horo_to_end_visib} also follows from a similar argument
using the visibility property. Alternatively, one could also appeal to 
\cite[Theorem~1.8]{CM:2025}. 

\section{Proofs of Theorems}\label{S:proofs}
In this section, we will present the proofs of our main results. 
\smallskip

But we begin with the construction referred to in \Cref{rem:unb_loc_Grom_mod_not_Grom_hyp}. 
Consider a bounded convex domain $D\subset\C^d$ with smooth boundary that is of finite type
everywhere except at only one point. Note that, by 
\cite[Theorem~1.1]{ZimGromHypKobMetConvDomFinType}, $(D,\koba_{D})$ is not Gromov hyperbolic.
Now let the exceptional point be $0_{\C^d}$ and suppose
without loss of generality that the complex tangent space to $\bdy D$ at $0_{\C^d}$ is
$\{z_1=0\}$. Consider the biholomorphic map $F:\C^d\setminus\{z_1=0\}\to\C^d\setminus\{z_1=0\}$ 
given by $(z_1,z_2,\dots,z_d)\mapsto(1/z_1,z_2,\dots,z_d)$. Writing $\OM\defeq F(D)$, it is
clear that $\OM$ is an unbounded domain in $\C^d$ that is not Gromov hyperbolic relative to its
Kobayashi distance. Note that $F$ is defined as a biholomorphism in a neighbourhood of 
$\clos{D}\setminus \{0_{\C^d}\}$ and note also that $\bdy\OM=F(\bdy D\setminus\{0_{\C^d}\})$.
From these two facts and \cite[Theorem~1.3]{ZimGromHypKobMetConvDomFinType} it follows that $\OM$
is a local Gromov model domain. This concludes the construction; we now turn to the proofs of the
results.

\begin{proof}[Proof of {\Cref{prp:loc_Grom_mod_asymp}}]
Using the fact that $\OM$ is a local Gromov model domain, choose a bounded neighbourhood $U$ of 
$\xi$ such that $U\cap\OM$ is connected, complete hyperbolic, Gromov hyperbolic, and such that 
$\mathsf{id}_{U\cap\OM}$ extends to a homeomorphism from $\clos{U\cap\OM}^G$ to $\clos{U\cap\OM}$.
Also choose a neighbourhood $W$ of $\xi$ such that $W\Subset U$. Since $\gamma$ and $\sigma$ land at
$\xi$, there exists $T\in [0,\infty)$ such that, for all $t\ge T$, $\gamma(t),\sigma(t)\in W\cap\OM$.
Note that, since $U\cap\OM$ is a Gromov model domain, in particular, every two points of 
$\bdy\OM\cap U$ possess the (weak) visibility property with respect to $\koba_{U\cap\OM}$
(see, for instance, \cite[Theorem~2.6]{ZimCharDomLimAut} and \cite[Corollary~3.2]{CMS2023}). 
Therefore, we may invoke \cite[Theorem~1]{ADS2023} to conclude that there exists $C<\infty$ such that
\begin{equation}\label{E:local}
\forall\,z,w\in W\cap\OM,\; \koba_\OM(z,w) \le \koba_{U\cap\OM}(z,w) \le \koba_{\OM}(z,w)+C.  
\end{equation}
This readily implies that $\gamma|_{[T,\infty)}$ and $\sigma|_{[T,\infty)}$ are continuous 
$(1,C)$-quasi-geodesics with respect to $\koba_{U\cap\OM}$. Now consider
$\widehat{\gamma}:[0,\infty)\lraw\OM$ and $\widehat{\sigma}:[0,\infty)\lraw\OM$ defined by
\[
\widehat{\gamma}(t)=\gamma(t+T), \ \ \ \widehat{\sigma}(t)=\sigma(t+T) \ \ \ \forall t\ge0. 
\]
Clearly both $\widehat{\gamma}$, $\widehat{\sigma}$ are $(1, C)$-quasi-geodesic rays in
$(U\cap\OM, \koba_{U\cap\OM})$ landing at $\xi\in\bdy(U\cap\OM)$. Since $(U\cap\OM, \koba_{U\cap\OM})$ is a 
Gromov model domain, appealing to Lemma~5.8 in \cite{AFGG}, we get 
\[
\sup_{t\geq 0}\koba_{U\cap\OM}(\widehat{\gamma}(t), \widehat{\sigma}(t))<\infty. 
\]
This together with \eqref{E:local} implies that 
\[
\sup_{t\geq 0}\koba_{\OM}(\widehat{\gamma}(t), \widehat{\sigma}(t))<\infty, 
\]
from which the desired conclusion follows easily. 
\end{proof}

We now present the proof of {\Cref{thm:loc_Grom_hyp_then_str_asymp}}. 
Our proof closely follows the proof of Theorem~4.3 in 
\cite{AFGG}. The main point to note is that,
under the hypotheses of our theorem, the assumption of 
boundedness in \cite[Theorem~4.3]{AFGG} is redundant.
Keeping this in mind, and for the sake of 
completeness, we first state a lemma, which
is essentially a restatement of \cite[Lemma~4.4]{AFGG}.
To state this lemma we need to make some 
preliminary remarks, which mirror those made in \cite{AFGG} 
immediately before the statement of Lemma~4.4.
\smallskip

With $\OM$ an arbitrary domain that admits squeezing functions, suppose we have a sequence
$(\gamma_n : [-T_n,T_n] \to \OM)_n$ of geodesic segments in $(\OM,\koba_{\OM})$ such that 
$T_n\to\infty$.
Suppose that there exist a sequence $(z_n)$ in $\OM$ and $C<\infty$ such that, for all $n$, 
$\koba_{\OM}(z_n,\gamma_n(0))\le C$ and such that $s_{\OM}(z_n)\to 1$. Write $r_n\defeq s_\OM(z_n)$ 
and, using the definition of the squeezing function, choose, for every $n\in\posint$, an injective 
holomorphic map $\phi_n:\OM\to\ball^d$ such that $\phi_n(z_n)=0$ and $B(0;r_n)\subset \OM_n\defeq 
\phi_n(\OM)$. Let
$\wh{\gamma}_n \defeq \phi_n \circ \gamma_n$; then $\wh{\gamma}_n$ is a geodesic segment in 
$(\OM_n,\koba_{\OM_{n}})$. Since $(\OM_n)$ expands to $\ball^d$, it is easy to see,
by the Arzel{\`a}--Ascoli theorem, that there exists a subsequence of 
$(\wh{\gamma}_n)$ that converges uniformly on the compact subsets of $\R$ to a 
$\koba_{\ball^d}$-geodesic-line $\eta:\R\to\ball^d$; we may
suppose, without loss of generality, that $(\wh{\gamma}_n)$ itself converges uniformly on the compact
subsets of $\R$ to the geodesic line $\eta$. The next lemma allows us to calculate 
$\koba_{\ball^d}(0,\range\,\eta)$ in terms of $\koba_{\OM_n}(0,\range\,\wh{\gamma}_n)$.

\begin{lemma}
In the situation described above, 
\[\koba_{\ball^d}(0,\range\,\eta) = \lim_{n\to\infty} \koba_{\OM_n}(0,\range\,\wh{\gamma}_n). \]
\end{lemma}
\noindent We omit the proof because it is the same, word for word, as the proof of Lemma~4.4 in \cite{AFGG}.

\begin{proof}[Proof of {{\Cref{thm:loc_Grom_hyp_then_str_asymp}}}]
Assume, to get a contradiction, that the statement is false. Then there exist $\xi\in\bdy\OM$ and
there exist $\koba_\OM$-geodesic-rays $\gamma$ and $\sigma$ in $\OM$ landing at $\xi$ that are not
strongly asymptotic. By \Cref{prp:loc_Grom_mod_asymp}, $\gamma$ and $\sigma$ are 
$\koba_\OM$-asymptotic. From this point one can follow the proof of \cite[Theorem~4.3]{AFGG} 
{\it verbatim} to obtain the required result.
\end{proof}

\begin{proof}[Proof of {\Cref{thm:strongly_asymp_then_cont_xtnsn}}]
It is easy to see that $\Psi$ is well-defined: if $\xi\in\bdy\OM$ and $\gamma,\sigma$ are geodesic
rays in $(\OM,\koba_\OM)$ landing at $\xi$, then by hypothesis $\gamma$ and $\sigma$ are strongly
asymptotic, and therefore by \Cref{res:geod_str_asymp_Buse_eq}, we conclude that
$B_\gamma=B_\sigma$, whence, a fortiori, $[B_\gamma]=[B_\sigma]$.
This also means that there is a
well-defined function $f_\xi:\OM\times\OM\to\R$ such that, for every geodesic ray $\gamma$ in
$(\OM,\koba_\OM)$ landing at $\xi$, and for every two points $z,p\in\OM$, $f_\xi(z,p)=B_\gamma(z,p)$.
\smallskip

To show that $\Psi$ is continuous, it is enough to prove that if 
$(x_n)\subset\OM$ is such that $x_n\to\xi\in\bdy\OM$
then $(\Psi(x_n))$ converges to $\Psi(\xi)$. To do this, 
it suffices to prove that for a fixed but arbitrary $p\in\OM$, the sequence of functions
$\big( \koba_\OM(\cdot\,,x_n)-\koba_\OM(x_n,p) \big)$ converges to the function $B_\gamma(\cdot\,,p)$,
where $\gamma$ is {\em some} geodesic ray in $(\OM,\koba_\OM)$ landing at $\xi$.
Observe that for $n\in\posint$ and $y,z\in X$ arbitrary,
\[ |\big(\koba_\OM(z,x_n)-\koba_\OM(x_n,p)\big)-\big(\koba_\OM(y,x_n)-\koba_\OM(x_n,p)\big)| \le 
\koba_\OM(z,y), \]
which shows that the sequence of functions considered is uniformly 1-Lipschitz. Since it is also clearly
pointwise relatively compact, 
it follows, by the Arzel{\`a}--Ascoli theorem, that in order to prove that the sequence
$\big( \koba_\OM(\cdot\,,x_n)-\koba_\OM(x_n,p) \big)$ converges (in $C(\OM,\R)$) to 
$f_\xi(\cdot\,,p)$, 
it suffices to prove that it converges pointwise to the said function, i.e.,
for every $z\in\OM$, $\koba_\OM(z,x_n)-\koba_\OM(x_n,p) \to f_\xi(z,p)$. Fix $z\in\OM$. We 
only need to prove that every subsequence of $(\koba_\OM(z,x_n)-\koba_\OM(x_n,p))$ has a further
subsequence that converges to $f_\xi(z,p)$; for notational convenience, we show that
$(\koba_\OM(z,x_n)-\koba_\OM(x_n,p))$ has a subsequence that converges to $f_\xi(z,p)$.
For every $n$, choose geodesic segments $\gamma_n$ and $\sigma_n$
in $(\OM, \koba_\OM)$ joining 
$z$ to $x_n$ and $p$ to $x_n$, respectively. Since $\OM$ is a complete hyperbolic visibility domain,
we may invoke \cite[Lemma~2.16]{CGMS2024} to conclude that $(\gamma_n)$ and $(\sigma_n)$ have
subsequences that converge uniformly on the compact subsets of $[0,\infty)$ to geodesic rays 
$\wt{\gamma}$ and $\wt{\sigma}$ that emanate from $z$ and $p$, respectively, and land at $\xi$. By
passing to subsequences successively, we may assume that $(\gamma_n)$ and $(\sigma_n)$ themselves
converge to $\wt{\gamma}$ and $\wt{\sigma}$, respectively. By hypothesis, $\wt{\gamma}$ and 
$\wt{\sigma}$ are strongly asymptotic, i.e., there exists $T\in\R$ such that
\[ \lim_{t\to\infty} \koba_\OM(\wt{\gamma}(t),\wt{\sigma}(t+T)) = 0. \]

Fix $t\in [0,\infty)$. For all $n$ sufficiently large
\begin{align*}
\koba_\OM(z,x_n) &= \koba_\OM(z,\gamma_n(t)) + \koba_\OM(\gamma_n(t),x_n), \\
\koba_\OM(x_n,p) &\le \koba_\OM(x_n,\gamma_n(t)) + \koba_\OM(\gamma_n(t),p) \\
\koba_\OM(x_n,p) &= \koba_\OM(x_n,\sigma_n(t+T)) + \koba_\OM(\sigma_n(t+T),p) \\
                 &\ge \koba_\OM(x_n,\gamma_n(t)) + \koba_\OM(\gamma_n(t),p) - 
                 2\koba_\OM(\gamma_n(t),\sigma_n(t+T)).
\end{align*}
From the first and second (in)equalities above,
\begin{equation} \label{eqn:lb_diff_def_horofunc1}
\koba_\OM(z,x_n)-\koba_\OM(x_n,p) \ge \koba_\OM(z,\gamma_n(t)) - \koba_\OM(\gamma_n(t),p).
\end{equation}
Similarly, from the first and third inequalities,
\begin{equation} \label{eqn:ub_diff_def_horofunc1}
\koba_\OM(z,x_n)-\koba_\OM(x_n,p) \le \koba_\OM(z,\gamma_n(t)) - \koba_\OM(\gamma_n(t),p) + 
2 \koba_\OM(\gamma_n(t),\sigma_n(t+T)).
\end{equation}
Recalling that $t$ is fixed, take $n\to\infty$ in \eqref{eqn:lb_diff_def_horofunc1} to get
\begin{equation} \label{eqn:lb_for_lim}
\liminf_{n\to\infty} \big( \koba_\OM(z,x_n)-\koba_\OM(x_n,p) \big) \ge \koba_\OM(z,\wt{\gamma}(t))-
\koba_\OM(\wt{\gamma}(t),p).
\end{equation}
Similarly, taking $n\to\infty$ in \eqref{eqn:ub_diff_def_horofunc1} we get
\begin{equation} \label{eqn:ub_for_lim}
\limsup_{n\to\infty} \big( \koba_\OM(z,x_n)-\koba_\OM(x_n,p) \big) \le \koba_\OM(z,\wt{\gamma}(t))-
\koba_\OM(\wt{\gamma}(t),p) + 2\koba_\OM(\wt{\gamma}(t),\wt{\sigma}(t+T)). 
\end{equation}
Now recall that $t$ itself was arbitrary, that 
\[ \lim_{t\to\infty} \big( \koba_\OM(z,\wt{\gamma}(t))-\koba_\OM(\wt{\gamma}(t),p) \big) \]
exists and equals $f_\xi(z,p)$, and that 
\[ \lim_{t\to\infty} \koba_\OM(\wt{\gamma}(t),\wt{\sigma}(t+T)) = 0 \]
to conclude, from \eqref{eqn:lb_for_lim} and \eqref{eqn:ub_for_lim}, that
\[ \lim_{n\to\infty} \big( \koba_\OM(z,x_n)-\koba_\OM(x_n,p) \big) = f_\xi(z,p). \]
By the discussion above, it follows that $\big( \koba_\OM(\cdot\,,x_n)-\koba_\OM(x_n,p) \big)$
converges, in the topology of $C(\OM,\R)$, to $f_\xi(\cdot,p)$. This establishes the continuity of 
$\Psi$. 
\smallskip

That $\Psi:\clos{\Omega}\to \clos{\OM}^H$ is an embedding, 
is a consequence of \Cref{res:if_conv_same_horo_then_Grom_equiv}.  
Note that, to prove the assertion about the embedding,
it suffices to check that if $(x_n)\subset \clos{\Omega}$ and $y\in\clos{\Omega}$ 
are such that
$\Psi(x_n)\to \Psi(y)$, then $x_n\to y$.  
The verification proceeds by considering the possible locations of $x_n$ and $y$ 
(in the interior or on the boundary) and whether $(x_n)$ is bounded.  
We illustrate the argument for one representative case: $x_n\in\bdy\Omega$ and $y\in\bdy\Omega$.
\smallskip
 
Since $(\Omega,\koba_\Omega)$ is complete and has the geodesic visibility property, 
\cite[Theorem~3.3]{CGMS2024} implies that $\clos{\Omega}^{End}$ is sequentially compact.  
Hence, some subsequence $(x_{k_n})$ converges to a point $\xi\in\bdy\clos{\Omega}^{End}$; 
it remains to show that $\xi=y$.  
Using the first countability of $\clos{\Omega}^{End}$ and the metrizability of $\clos{\tau(\Omega)}$ 
(where $\tau$ is as in \Cref{S:Prelims} for $(\Omega,\koba_\Omega)$), 
one can find a sequence $(x'_n)\subset\Omega$ such that 
$x'_n\to\xi$ and $(\Psi(x_{k_n}))$ and $(\Psi(x'_n))$ converge to
the same point (hence to $\Psi(y)$).  
Let $(y_n)\subset\Omega$ be such that $y_n\to y$. 
Since $\Psi|_{\Omega}=\tau$, we have two sequences $(x'_n)$ and $(y_n)$ in $\Omega$ with 
$\tau(x'_n)$ and $\tau(y_n)$ converging to the same element of $\clos{\tau(\Omega)}$.  
By \Cref{res:if_conv_same_horo_then_Grom_equiv},
\[
\lim_{n\to\infty} (x'_n|y_n)_o = \infty,
\]
where $o\in\OM$ is a fixed but arbitrary point. If $\xi\neq y$, then by a standard fact for domains 
with the visibility property (see, e.g.,
\cite[Prop.~2.4]{BNT2022}, \cite[Prop.~3.1]{CMS2023}, \cite[Lem.~2.15]{CGMS2024}),  
\[
\limsup_{n\to\infty} (x'_n|y_n)_o < \infty,
\]
a contradiction. Thus $\xi=y$, completing this case.  
The remaining cases can be handled similarly.
\end{proof}

Now we are ready to prove \Cref{prp:sq_fn_tend_1_bdry_cond1}.

\begin{proof}[Proof of {\Cref{prp:sq_fn_tend_1_bdry_cond1}}]
Given $p\in\bdy\OM$, let $\gamma$ be the connected component of $\bdy\OM$ containing $p$, and let 
$K_\gamma$ be the connected component of $\C\setminus\OM$ containing $\gamma$ (or, equivalently,
$p$). Then we know, by \cite[Lemma~4.11]{CGMS2024}, that there exists a neighbourhood $V$ of $p$ in 
$\C$ such that 
\begin{equation} \label{eqn:prop_nbd_p} 
V=(V\cap\OM)\cup (V\cap K_\gamma). 
\end{equation}
Consider the domain 
\[
\OM_\gamma \defeq \C_\infty \setminus \cl_{\C_\infty}(K_\gamma) 
\]
(as a set, this is equal to $\C\setminus K_\gamma$ when $K_\gamma$ is unbounded as a subset of $\C$ and
is equal to $\C_\infty\setminus K_\gamma$ when $K_\gamma$ is bounded; the point to be emphasised is that
we always regard $\OM_\gamma$ as being a domain in $\C_\infty$).
Then $\OM_\gamma$ is a simply connected domain
in $\C_\infty$ containing $\OM$ whose boundary is precisely $\cl_{\C_\infty}(\gamma)$. Now, by 
\cite[Proposition~4.14]{CGMS2024}, $\cl_{\C_\infty}(\gamma)$ is a Jordan curve. Therefore, if we choose
and fix a biholomorphism $\phi_\gamma:\OM_\gamma\to \unitdisk$,
then, by Carath{\'e}odory's theorem on the extension of biholomorphisms (see, for instance,
\cite[Theorem~4.3.1]{BCM}), 
$\phi_\gamma$ will extend to a homeomorphism, which we will continue to
call $\phi_\gamma$, from $\cl_{\C_\infty}(\OM_\gamma)$ to $\clos{\unitdisk}$. We may assume, by 
post-composing with a suitable automorphism of $\unitdisk$, that $\phi_\gamma(p)=1$.
Note that $V\cap\clos{\OM}$
is a neighbourhood of $p$ in $\clos{\OM}$; so $W\defeq \phi_\gamma(V\cap\clos{\OM})$
is a neighbourhood of $1$
in $\clos{\unitdisk}$. Note that by \eqref{eqn:prop_nbd_p}, 
\[ W\cap\unitdisk = W\cap\phi_\gamma(\OM_\gamma)=W\cap\phi_\gamma(\OM). \]
By the biholomorphic invariance of the squeezing function, writing $\wt{\OM}\defeq \phi_\gamma(\OM)$,
we have
\[ s_{\wt{\OM}}\circ\phi_\gamma=s_\OM. \]
Therefore, in order to prove that $\lim_{\OM\,\ni\,z\to p}s_\OM(z)=1$, it suffices to prove that
\begin{equation} \label{eqn:lim_sq_func_OM-tld_1_1}  
\lim_{W\cap\unitdisk\ni z\to 1}s_{\wt{\OM}}(z)=1. 
\end{equation}
Since $W$ is a neighbourhood of $1$ in $\clos{\unitdisk}$, it follows, from the properties of the
Poincar{\'e} distance, which is the Kobayashi distance on $\unitdisk$, that 
\[ \lim_{z\to 1} \koba_\unitdisk(z,\unitdisk\setminus W) = \infty. \]
Since $\unitdisk\setminus\wt{\OM}\subset\unitdisk\setminus W$, it also follows that
\[ \lim_{z\to 1}\koba_\unitdisk(z,\unitdisk\setminus\wt{\OM}) = \infty. \]
Therefore, given $r<1$, there exists a neighbourhood $\wh{W}$ of $1$ in $\clos{\unitdisk}$ such that
\begin{equation} \label{eqn:Kob_dist_1_OM-tld-cmplmnt_large}  
\forall\,z\in\wh{W}\cap\unitdisk,\; \koba_{\unitdisk}(z,\unitdisk\setminus\wt{\OM}) > 
\koba_\unitdisk(0,r) = \tanh^{-1}(r) = (1/2) \log\left(\frac{1+r}{1-r}\right). 
\end{equation}
For any $z\in\wh{W}\cap\unitdisk$, we choose an automorphism $T_z$ of $\unitdisk$ such that $T_z(z)=0$.
Note that, by the biholomorphic invariance of the Poincar{\'e} distance and 
\eqref{eqn:Kob_dist_1_OM-tld-cmplmnt_large}, 
\[ \koba_\unitdisk(0,\unitdisk\setminus T_z(\wt{\OM})) = \koba_\unitdisk(z,\unitdisk\setminus\wt{\OM})
> \koba_\unitdisk(0,r) = \koba_\unitdisk(0,\unitdisk\setminus D(0,r));\]
therefore, $D(0;r)\subset T_z(\wt{\OM})$. Hence $s_{\wt{\OM}}(z)\ge r$. Since $r$ was arbitrary, this
proves \eqref{eqn:lim_sq_func_OM-tld_1_1} and with it the proposition.
\end{proof}

\begin{corollary}
Suppose that $\OM\subset\C$ is a domain that satisfies Condition~1. Then the mapping $\Psi:\clos{\OM}\to
\clos{(\OM,\koba_\OM)}^H$ given by 
\[ \Psi(\xi) \defeq \begin{cases}
    [\koba_\OM(\cdot\,,\xi)], &\text{if } \xi\in\OM,\\
    [B_\gamma], &\text{if } \xi\in\bdy\OM \text{ and } \gamma \text{ is some geodesic ray in } (\OM,\koba_\OM)
    \text{ landing at } \xi,
\end{cases} \]
is well-defined and an embedding of $\clos{\OM}$ in $\clos{(\OM,\koba_\OM)}^H$.
\end{corollary}

\begin{proof}
First note that, by the discussion following \Cref{Def:cond1}, $\OM$ is a local Gromov model 
domain. Therefore, by 
\Cref{prp:loc_Grom_mod_asymp}, for every $\xi\in\bdy\OM$, and for every two $\koba_\OM$-geodesic rays
$\gamma$ and $\sigma$ landing at $\xi$, $\gamma$ and $\sigma$ are $\koba_\OM$-asymptotic. Using
\Cref{prp:sq_fn_tend_1_bdry_cond1} and \Cref{thm:loc_Grom_hyp_then_str_asymp}, we conclude that
$\gamma$ and $\sigma$ are strongly asymptotic. Therefore we may invoke 
\Cref{thm:strongly_asymp_then_cont_xtnsn} to draw the required conclusion.
\end{proof}

\bigskip

{\bf Acknowledgements:} Sushil Gorai is partially supported by a Core Research Grant
(CRG/2022/003560) from Science and Engineering Research Board, Department of Science and 
Technology, Government of India. Anwoy Maitra is partially
supported by a New Faculty Initiation Grant
(SB24250650MANFIG009155) from the Indian Institute of
Technology Madras.

\end{document}